\documentclass{ndjflart}
\usepackage[T1]{fontenc}
\usepackage{tgtermes}
\usepackage{mathptmx}  
\usepackage[scaled=.92]{helvet}
\usepackage{amsthm,amsmath,amssymb}
\usepackage{mathrsfs}
\usepackage[colorlinks,citecolor=blue,urlcolor=blue]{hyperref}  

\usepackage{enumerate}

\artstatus{am} 
\date{June 22, 2012}

\startlocaldefs

\endlocaldefs

\newcommand{\Sh}{\ensuremath{\protect{S_{\st{}}}}}

\newcommand{\INT}{\ensuremath{{\usftext{int}}}}

\usepackage[english]{babel}
\usepackage{amssymb,amsmath,hyperref, amsthm}
\usepackage{amsrefs}
\usepackage{mathrsfs}
\usepackage{cleveref, enumerate}

\newtheorem{thm}{Theorem}

\newtheorem{cor}[thm]{Corollary}
\newtheorem{defi}[thm]{Definition}
\newtheorem{rem}[thm]{Remark}
\newtheorem{nota}[thm]{Notation}

\newtheorem{princ}[thm]{Principle}

\newtheorem{tempo*}[thm]{Template}

\newcommand\be{\begin{equation}}
\newcommand\ee{\end{equation}}

\newbox\gnBoxA
\newdimen\gnCornerHgt
\setbox\gnBoxA=\hbox{$\ulcorner$}
\global\gnCornerHgt=\ht\gnBoxA
\newdimen\gnArgHgt

\def\Godelnum #1{%
	\setbox\gnBoxA=\hbox{$#1$}%
	\gnArgHgt=\ht\gnBoxA%
	\ifnum \gnArgHgt<\gnCornerHgt
		\gnArgHgt=0pt%
	\else
		\advance \gnArgHgt by -\gnCornerHgt%
	\fi
	\raise\gnArgHgt\hbox{$\ulcorner$} \box\gnBoxA %
		\raise\gnArgHgt\hbox{$\urcorner$}}
\def\bdefi{\begin{defi}\rm}
\def\edefi{\end{defi}}
\def\bnota{\begin{nota}\rm}
\def\enota{\end{nota}}
\def\brem{\begin{rem}\rm}
\def\erem{\end{rem}}

\newcommand{\tup}{\underline}

\def\IST{\textup{\textsf{IST}}}

\def\DNR{\textup{DNR}}

\def\H{\textup{\textsf{H}}}
\def\RCA{\textup{\textsf{RCA}}}

\def\RCAo{\textup{\textsf{RCA}}_{0}^{\omega}}
\def\RCAO{\textup{\textsf{RCA}}_{0}^{\Lambda}}

\def\T{\mathcal{T}}

\def\bye{\end{document}}

\def\P{\textup{\textsf{P}}}

\def\N{{\mathbb  N}}

\def\({\textup{(}}
\def\){\textup{)}}

\def\st{\textup{st}}
\def\asa{\leftrightarrow}

\def\di{\rightarrow}

\def\ACA{\textup{\textsf{ACA}}}
\def\paai{\Pi_{1}^{0}\textup{-\textsf{TRANS}}}

\def\QFAC{\textup{\textsf{QF-AC}}}

\def\RS{{\mathfrak{RS}}}

\def\HGMP{\textup{\textsf{HGMP}}}
\def\HIP{\textup{\textsf{HIP}}}
\def\HIO{\textup{\textsf{HIO}}}

\def\MU{\textup{\textsf{MU}}}

\def\HAC{\textup{\textsf{HAC}}}

\def\INT{\textup{\textsf{int}}}

\def\DNR{\textup{\textsf{DNR}}}

\def\UDNR{\textup{\textsf{UDNR}}}

\numberwithin{equation}{section}
\numberwithin{thm}{section}

\def\NCS{\textup{\textsf{NCS}}}
\def\UNCS{\textup{\textsf{UNCS}}}
\def\AST{\textup{\textsf{AST}}}

\def\KPT{\textup{\textsf{KPT}}}
\def\UKPT{\textup{\textsf{UKPT}}}
\def\UOG{\textup{\textsf{U1G}}}

\def\UHYP{\textup{\textsf{UHYP}}}
\def\AMT{\textup{\textsf{AMT}}}

\def\QFAC{\textup{\textsf{QF-AC}}}

\def\SADS{\textup{\textsf{SADS}}}
\def\USADS{\textup{\textsf{USADS}}}

\def\PG{\Pi_{1}^{0}\textup{\textsf{G}}}
\def\UPG{\textup{\textsf{U}}\Pi_{1}^{0}\textup{\textsf{G}}}

\def\UDNR{\textup{\textsf{UDNR}}}

\begin{document}
\begin{frontmatter}

  \title{Refining the taming of the Reverse Mathematics zoo}

  \author{\fnms{Sam} 
    \snm{Sanders}
    \corref{}
    \ead[label=e1]{sasander@me.com}
    \ead[label=u1,url]{http://sasander.wix.com/academic}
  }
  \address{Department of Mathematics, Ghent University \\
    Krijgslaan 281, 9000 Gent, BELGIUM \\
    ~\\
    \emph{and}\\
    ~\\
    Munich Center for Mathematical Philosophy, LMU Munich, \\
    Geschwister-Scholl-Platz 1, 80539 Munich, GERMANY\\
    ~\\
    \printead{e1}\\
    \printead{u1} }%

  \runauthor{S.~Sanders}

\begin{abstract}
\emph{Reverse Mathematics} is a program in the foundations of mathematics.  It provides an elegant classification in which the majority of theorems of ordinary mathematics fall into \emph{only five} categories, based on the `Big Five' logical systems.  
Recently, a lot of effort has been directed towards finding \emph{exceptional} theorems, i.e.\ which fall \emph{outside} the Big Five.  
The so-called Reverse Mathematics zoo is a collection of such exceptional theorems (and their relations).    
It was shown in \cite{samzoo} that a number of \emph{uniform} versions of the zoo-theorems, i.e.\ where a functional computes the objects stated to exist, fall in the third Big Five category \emph{arithmetical comprehension}, inside Kohlenbach's {higher-order} Reverse Mathematics.  
In this paper, we extend and refine the results from \cite{samzoo}.  
In particular, we establish analogous results for recent additions to the Reverse Mathematics zoo, thus establishing that the latter disappear at the uniform level. 
Furthermore, we show that the aforementioned equivalences can be proved using only intuitionistic logic.  
Perhaps most surprisingly, these \emph{explicit} equivalences are extracted from \emph{nonstandard} equivalences in Nelson's \emph{internal set theory}, and we show that the nonstandard equivalence can be recovered from the explicit ones.  Finally, the following zoo-theorems are studied in this paper:  $\Pi^0_1\textsf{G}$ (existence of uniformly $\Pi^0_1$-generics), $\textsf{FIP}$ (finite intersection principle), \textsf{1-GEN} (existence of 1-generics), \textsf{OPT} (omitting partial types principle), \textsf{AMT} (atomic model theorem), \textsf{SADS} (stable ascending or descending sequence), \textsf{AST} (atomic model theorem with sub-enumerable types), \textsf{NCS} (existence of non-computable sets), and \textsf{KPT} (Kleene/Post theorem that there exist Turing incomparable sets).        
\end{abstract}

\begin{keyword}[class=AMS]
  \kwd[Primary ]{03B30} \kwd{03F35} \kwd[; Secondary ]{26E35}
\end{keyword}

\begin{keyword}
\kwd{higher-order Reverse Mathematics} \kwd{Reverse Mathematics zoo} \kwd{Nonstandard Analysis}
\end{keyword}

\end{frontmatter}

%
\maketitle
\thispagestyle{empty}

Accepted for publication in the \emph{Notre Dame Journal of Formal Logic} (2016).  
\section{Introduction: Reverse Mathematics and its zoo}\label{intro}
The subject of this paper is the \emph{Reverse Mathematics} classification in Kohlenbach's framework (\cite{kohlenbach2}) of uniform versions of principles from the \emph{Reverse Mathematics zoo} (\cite{damirzoo}), namely as equivalent to \emph{arithmetical comprehension}.  A number of theorems from the \emph{Reverse Mathematics zoo} have been classified in this way in \cite{samzoo}, and this paper continues \emph{and refines} this classification.     
We first discuss the aforementioned italicised notions in more detail.

\medskip

First of all, an overview of the foundational program Reverse Mathematics (RM for short), may be found in \cites{simpson2, simpson1}.  
Perhaps \emph{the} main conceptual result of RM is that the majority of theorems from \emph{ordinary mathematics}, i.e.\ about countable and separable objects, fall into \emph{only five} 
classes of which the associated logical systems have been christened `the Big Five' (See e.g.\ \cite{montahue}*{p.\ 432} and \cite{dslice}*{p.\ 69} for this point of view).  
Recently, considerable effort has been spent identifying theorems falling \emph{outside} of the Big Five systems.  
For an overview, exceptional theorems (and their relations) falling below the third Big Five system $\ACA_{0}$, are collected in Dzhafarov's so-called RM zoo (\cite{damirzoo}).   

\medskip

It was established in \cite{samzoo} that a number of exceptional principles inhabiting the RM zoo become \emph{non-exceptional at the uniform level}, namely that the uniform versions of RM zoo-principles are all equivalent to arithmetical comprehension, the aforementioned third Big Five system of RM. 
As an example of such a `uniform version', consider the principle \textsf{UDNR} from \cite{samzoo}*{\S3}.  
\be\label{UDNR2}\tag{\textsf{UDNR}}
(\exists \Psi^{1\di1})\big[(\forall A^{1})(\forall e^{0})(\Psi(A)(e)\ne \Phi_{e}^{A}(e))\big].
\ee
Clearly, $\UDNR$ is the uniform version of the zoo principle\footnote{We sometimes refer to inhabitants of the RM zoo as `theorems' and sometimes as `principles'.} $\DNR$, defined as:  
\be\label{DNR2}\tag{\textup{\DNR}}
(\forall A^{1})(\exists f^{1})(\forall e^{0})\big[f(e)\ne \Phi_{e}^{A}(e)\big].
\ee  
Now, the principle $\DNR$ was introduced in \cite{withgusto} and is strictly weaker than $\textsf{WWKL}$ (See \cite{compdnr}) where the latter principle sports \emph{a small number} of Reverse Mathematics equivalences (\cites{montahue, yuppie, yussie}), but is not counted as a `Big Five' system.  
The exceptional status of $\DNR$ notwithstanding, it was shown in \cite{samzoo}*{\S3} that $\UDNR\asa (\exists^{2})$, where the second principle is the functional version of arithmetical comprehension, the third Big Five system of RM, defined as follows:
\be\tag{$\exists^{2}$}
(\exists \varphi^{2})(\forall f^{1})\big(\varphi(f)=0\asa (\exists n)f(n)\ne 0   \big).  
\ee
In other words, \emph{the `exceptional' status of $\DNR$ disappears completely if we consider its uniform version $\UDNR$}.  
Furthermore, the proof of the equivalence $\UDNR\asa (\exists^{2})$ takes place in $\RCAo$ (See Section \ref{base}), 
the base theory of Kohlenbach's \emph{higher-order Reverse Mathematics}.  This system is a conservative extension of $\RCA_{0}$, the usual base theory of RM, for the second-order language.    

\medskip

More generally, a number of uniform zoo-principles are shown to be equivalent to arithmetical comprehension over $\RCAo$ in \cite{samzoo}.  
A general template for classifying (past and future) zoo-principles in the same way was also formulated in the latter.    
In Section \ref{CLASS}, we show that this template works for a number of new theorems from the RM zoo, and refine the associated results considerably, as discussed next.       

\medskip

The methodology by which the aforementioned equivalences are obtained, constitutes somewhat of a surprise:
In particular, the equivalences in this paper are formulated as theorems of Kohlenbach's base theory $\RCAo$ (See \cite{kohlenbach2} and Section~\ref{popo}), 
but are obtained by applying the algorithm $\RS$ (See Section~\ref{tempie}) to associated equivalences \emph{in Nonstandard Analysis}, in particular Nelson's \emph{internal set theory} (See \cite{wownelly} and Section \ref{fooker}).  
Besides providing a streamlined and uniform approach, the use of Nonstandard Analysis via $\RS$ also results in \emph{explicit\footnote{An implication $(\exists \Phi)A(\Phi)\di (\exists \Psi)B(\Psi)$ is \emph{explicit} if there is a term $t$ in the language such that additionally $(\forall \Phi)[A(\Phi)\di B(t(\Phi))]$, i.e.\ $\Psi$ can be explicitly defined in terms of $\Phi$.}} equivalences \emph{without extra effort}.  In particular, we shall just prove equivalences inside Nonstandard Analysis \emph{without paying any attention to effective content}, 
and extract the explicit equivalences using the algorithm $\RS$.   This new `computational aspect' of Nonstandard Analysis is perhaps the true surprise of our taming of the RM zoo.  

\medskip

The following zoo-theorems are studied in Section \ref{CLASS} in the aforementioned way:  $\Pi^0_1\textsf{G}$ (existence of uniformly $\Pi^0_1$-generics), $\textsf{FIP}$ (finite intersection principle), \textsf{1-GEN} (existence of 1-generics), \textsf{OPT} (omitting partial types principle), \textsf{AMT} (atomic model theorem), \textsf{SADS} (stable ascending or descending sequence), \textsf{AST} (atomic model theorem with sub-enumerable types), \textsf{NCS} (existence of non-computable sets), and \textsf{KPT} (Kleene/Post theorem that there exist Turing incomparable sets).

\medskip

Furthermore, we shall refine the results from \cite{samzoo} and Section~\ref{CLASS} of this paper as follows in Section~\ref{refine}: 
First of all, while all results sketched above are proved using classical logic, we show in Section~\ref{culdesac} that they also go through for intuitionistic logic.  
Secondly, we formulate in Section~\ref{herbie} a special kind of explicit equivalence, called \emph{Herbrandisation}, from which we can re-obtain the \emph{original} equivalence in Nonstandard Analysis.  In other words, the \emph{Herbrandisation} is `meta-equivalent' to the nonstandard implication from which it was extracted.     

\medskip

In conclusion, this paper continues and refines the `taming of the RM zoo' initiated in \cite{samzoo}, i.e.\ we establish the equivalence between uniform RM zoo principles and arithmetical comprehension using intuitionistic logic.  
Furthermore, thanks to a new computational aspect of Nonstandard Analysis, we obtain `for free' \emph{explicit}\footnote{An implication $(\exists \Phi)A(\Phi)\di (\exists \Psi)B(\Psi)$ is \emph{explicit} if there is a term $t$ in the language such that additionally $(\forall \Phi)[A(\Phi)\di B(t(\Phi))]$, i.e.\ $\Psi$ can be explicitly defined in terms of $\Phi$.}  equivalences (not involving Nonstandard Analysis) from (non-effective) equivalences in Nonstandard Analysis, \emph{and vice versa}.  

\section{About and around internal set theory}\label{base}
In this section, we introduce Nelson's \emph{internal set theory}, first introduced in \cite{wownelly}, and its fragment $\P$ from \cite{brie}.
We shall also introduce Kohlenbach's base theory $\RCAo$ from \cite{kohlenbach2}, and the system $\RCAO$, which is based on $\P$.  
These systems are also introduced in \cite{samzoo}*{\S2}, but we include their definitions for completeness.     
\subsection{Introduction: Internal set theory}\label{fooker}
In Nelson's \emph{syntactic} approach to Nonstandard Analysis (\cite{wownelly}), as opposed to Robinson's semantic one (\cite{robinson1}), a new predicate `st($x$)', read as `$x$ is standard' is added to the language of \textsf{ZFC}, the usual foundation of mathematics.  
The notations $(\forall^{\st}x)$ and $(\exists^{\st}y)$ are short for $(\forall x)(\st(x)\di \dots)$ and $(\exists y)(\st(y)\wedge \dots)$.  A formula is called \emph{internal} if it does not involve `st', and \emph{external} otherwise.   
The three external axioms \emph{Idealisation}, \emph{Standard Part}, and \emph{Transfer} govern the new predicate `st';  they are introduced in Definition \ref{frackck} below, where the superscript `fin' in \textsf{(I)} means that $x$ is finite, i.e.\ its number of elements are bounded by a natural number.  
\bdefi[External axioms of $\IST$]\label{frackck}
\begin{enumerate}
\item[\textsf{(I)}] $(\forall^{\st~\textup{fin}}x)(\exists y)(\forall z\in x)\varphi(z,y)\di (\exists y)(\forall^{\st}x)\varphi(x,y)$, for internal $\varphi$ with any (possibly nonstandard) parameters.  
\item[\textsf{(S)}] $(\forall^{\st} x)(\exists^{\st}y)(\forall^{\st}z)(z\in y\asa (z\in y\wedge \varphi(z)))$, for any formula $\varphi$.  
\item[\textsf{(T)}] $(\forall^{\st}t)\big[(\forall^{\st}x)\varphi(x, t)\di (\forall x)\varphi(x, t)\big]$, where $\varphi$ is internal and only has free variables $t, x$.  
\end{enumerate}
\edefi
The system \textsf{IST} is (the internal system) \textsf{ZFC} extended with the aforementioned external axioms.  
Furthermore, $\IST$ is a conservative extension of \textsf{ZFC} for the internal language, as proved in \cite{wownelly}.    

\medskip

In \cite{brie}, the authors study G\"odel's system $\textsf{T}$ extended with special cases of the external axioms of \textsf{IST}.  
In particular, they consider nonstandard extensions of the (internal) systems \textsf{E-HA}$^{\omega}$ and $\textsf{E-PA}^{\omega}$, respectively \emph{Heyting and Peano arithmetic in all finite types and the axiom of extensionality}.       
We refer to \cite{brie}*{\S2.1} for the exact details of these (mainstream in mathematical logic) systems.  
We do mention that in these systems of higher-order arithmetic, each variable $x^{\rho}$ comes equipped with a superscript denoting its type, which is however often implicit.  
As to the coding of multiple variables, the type $\rho^{*}$ is the type of finite sequences of type $\rho$, a notational device used in \cite{brie} and this paper.  Underlined variables $\underline{x}$ consist of multiple variables of (possibly) different type.  

\medskip

In the next section, we introduce the system $\P$ assuming familiarity with the higher-type framework of G\"odel's system $\textsf{T}$ (See e.g.\ \cite{brie}*{\S2.1} for the latter).    
\subsection{The system $\P$}\label{popo}
In this section, we introduce the system $\P$.  We first discuss some of the external axioms studied in \cite{brie}.  
First of all, Nelson's axiom \emph{Standard part} is weakened to $\HAC_{\INT}$ as follows:
\be\tag{$\HAC_{\INT}$}
(\forall^{\st}x^{\rho})(\exists^{\st}y^{\tau})\varphi(x, y)\di (\exists^{\st}F^{\rho\di \tau^{*}})(\forall^{\st}x^{\rho})(\exists y^{\tau}\in F(x))\varphi(x,y),
\ee
where $\varphi$ is any internal formula.  Note that $F$ only provides a \emph{finite sequence} of witnesses to $(\exists^{\st}y)$, explaining its name \emph{Herbrandized Axiom of Choice}.      
Secondly,  Nelson's axiom idealisation \textsf{I} appears in \cite{brie} as follows:  
\be\tag{\textsf{I}}
(\forall^{\st} x^{\sigma^{*}})(\exists y^{\tau} )(\forall z^{\sigma}\in x)\varphi(z,y)\di (\exists y^{\tau})(\forall^{\st} x^{\sigma})\varphi(x,y), 
\ee
where $\varphi$ is again an internal formula.  
Finally, as in \cite{brie}*{Def.\ 6.1}, we have the following definition.
\bdefi\label{debs}
The set $\T^{*}$ is defined as the collection of all the constants in the language of $\textsf{E-PA}^{\omega*}$.  
The system $ \textsf{E-PA}^{\omega*}_{\st} $ is defined as $ \textsf{E-PA}^{\omega{*}} + \T^{*}_{\st} + \textsf{IA}^{\st}$, where $\T^{*}_{\st}$
consists of the following axiom schemas.
\begin{enumerate}
\item The schema\footnote{The language of $\textsf{E-PA}_{\st}^{\omega*}$ contains a symbol $\st_{\sigma}$ for each finite type $\sigma$, but the subscript is always omitted.  Hence $\T^{*}_{\st}$ is an \emph{axiom schema} and not an axiom.\label{omit}} $\st(x)\wedge x=y\di\st(y)$,
\item The schema providing for each closed term $t\in \T^{*}$ the axiom $\st(t)$.
\item The schema $\st(f)\wedge \st(x)\di \st(f(x))$.
\end{enumerate}
The external induction axiom \textsf{IA}$^{\st}$ is as follows.  
\be\tag{\textsf{IA}$^{\st}$}
\Phi(0)\wedge(\forall^{\st}n^{0})(\Phi(n) \di\Phi(n+1))\di(\forall^{\st}n^{0})\Phi(n).     
\ee
\edefi
For the full system $\P\equiv \textsf{E-PA}^{\omega*}_{\st} +\HAC_{\INT} +\textsf{I}$, we have the following theorem.  
Here, the superscript `$S_{\st}$' is the syntactic translation defined in \cite{brie}*{Def.\ 7.1}.    
\begin{thm}\label{consresult}
Let $\Phi(\tup a)$ be a formula in the language of \textup{\textsf{E-PA}}$^{\omega*}_{\st}$ and suppose $\Phi(\tup a)^\Sh\equiv\forall^{\st} \tup x \, \exists^{\st} \tup y \, \varphi(\tup x, \tup y, \tup a)$. If $\Delta_{\INT}$ is a collection of internal formulas and
\be\label{antecedn}
\P + \Delta_{\INT} \vdash \Phi(\tup a), 
\ee
then one can extract from the proof a sequence of closed terms $t$ in $\mathcal{T}^{*}$ such that
\be\label{consequalty}
\textup{\textsf{E-PA}}^{\omega*} + \Delta_{\INT} \vdash\  \forall \tup x \, \exists \tup y\in \tup t(\tup x)\ \varphi(\tup x,\tup y, \tup a).
\ee
\end{thm}
\begin{proof}
Immediate by \cite{brie}*{Theorem 7.7}.  
\end{proof}
It is important to note that the proof of the soundness theorem in \cite{brie}*{\S7} provides a \emph{term extraction algorithm} $\mathcal{A}$ to obtain the term $t$ from the theorem.  

\medskip

The following corollary is essential to our results.  We shall refer to formulas of the form $(\forall^{\st}\underline{x})(\exists^{\st}\underline{y})\psi(\underline{x},\underline{y}, \underline{a})$ for internal $\psi$ as (being in) \emph{the normal form}.    
\begin{cor}\label{consresultcor}
If for internal $\psi$ the formula $\Phi(\underline{a})\equiv(\forall^{\st}\underline{x})(\exists^{\st}\underline{y})\psi(\underline{x},\underline{y}, \underline{a})$ satisfies \eqref{antecedn}, then 
$(\forall \underline{x})(\exists \underline{y}\in t(\underline{x}))\psi(\underline{x},\underline{y},\underline{a})$ is proved in the corresponding formula \eqref{consequalty}.  
\end{cor}
\begin{proof}
Clearly, if for $\psi$ and $\Phi$ as given we have $\Phi(\underline{a})^{S_{\st}}\equiv \Phi(\underline{a})$, then the corollary follows immediately from the theorem.  
A tedious but straightforward verification using the clauses (i)-(v) in \cite{brie}*{Def.\ 7.1} establishes that indeed $\Phi(\underline{a})^{S_{\st}}\equiv \Phi(\underline{a})$.  
This verification is performed in full detail in \cite{samzoo}*{\S2} and \cite{sambo}.  
\end{proof}
Finally, the previous theorems do not really depend on the presence of full Peano arithmetic.  
Indeed, let \textsf{E-PRA}$^{\omega}$ be the system defined in \cite{kohlenbach2}*{\S2} and let \textsf{E-PRA}$^{\omega*}$ be its extension with types for finite sequences as in \cite{brie}*{\S2}.  
\begin{cor}\label{consresultcor2}
The previous theorem and corollary go through for $\P$ replaced by $\P_{0}\equiv \textsf{\textup{E-PRA}}^{\omega*}+\T_{\st}^{*} +\HAC_{\INT} +\textsf{\textup{I}}$.  
\end{cor}
\begin{proof}
The proof of \cite{brie}*{Theorem 7.7} goes through for any fragment of \textsf{E-PA}$^{\omega{*}}$ which includes \textsf{EFA}, sometimes also called $\textsf{I}\Delta_{0}+\textsf{EXP}$.  
In particular, the exponential function is (all what is) required to `easily' manipulate finite sequences.    
\end{proof}
Finally, we define $\RCAO$ as the system $\P_{0}+\QFAC^{1,0}$.  Recall that Kohlenbach defines $\RCAo$ in \cite{kohlenbach2}*{\S2} as \textsf{E-PRA}$^{\omega}+\QFAC^{1,0}$ 
where the latter is the axiom of choice limited to formulas $(\forall f^{1})(\exists n^{0})\varphi_{0}(f, n)$, $\varphi_{0}$ quantifier-free.     

\subsection{Notations and remarks}
We introduce some notations regarding $\RCAO$.  First of all, we shall mostly follow Nelson's notations as in \cite{bennosam}.  
\begin{rem}[Standardness]\label{notawin}\rm
As suggested above, we write $(\forall^{\st}x^{\tau})\Phi(x^{\tau})$ and also $(\exists^{\st}x^{\sigma})\Psi(x^{\sigma})$ as short for 
$(\forall x^{\tau})\big[\st(x^{\tau})\di \Phi(x^{\tau})\big]$ and $(\exists x^{\sigma})\big[\st(x^{\sigma})\wedge \Psi(x^{\sigma})\big]$.     
We also write $(\forall x^{0}\in \Omega)\Phi(x^{0})$ and $(\exists x^{0}\in \Omega)\Psi(x^{0})$ as short for 
$(\forall x^{0})\big[\neg\st(x^{0})\di \Phi(x^{0})\big]$ and $(\exists x^{0})\big[\neg\st(x^{0})\wedge \Psi(x^{0})\big]$.  Furthermore, if $\neg\st(x^{0})$ (resp.\ $\st(x^{0})$), we also say that $x^{0}$ is `infinite' (resp.\ `finite') and write `$x^{0}\in \Omega$'.  
Finally, a formula $A$ is `internal' if it does not involve `$\st$', and $A^{\st}$ is defined from $A$ by appending `st' to all quantifiers (except bounded number quantifiers).    
\end{rem}
Secondly, the notion of equality in $\RCAO$ is important to our enterprise.  
\begin{rem}[Equality]\label{equ}\rm
The system $\RCAo$ includes equality between natural numbers `$=_{0}$' as a primitive.  Equality `$=_{\tau}$' for type $\tau$-objects $x,y$ is defined as follows:
\be\label{aparth}
[x=_{\tau}y] \equiv (\forall z_{1}^{\tau_{1}}\dots z_{k}^{\tau_{k}})[xz_{1}\dots z_{k}=_{0}yz_{1}\dots z_{k}]
\ee
if the type $\tau$ is composed as $\tau\equiv(\tau_{1}\di \dots\di \tau_{k}\di 0)$.
In the spirit of Nonstandard Analysis, we define `approximate equality $\approx_{\tau}$' as follows:
\be\label{aparth2}
[x\approx_{\tau}y] \equiv (\forall^{\st} z_{1}^{\tau_{1}}\dots z_{k}^{\tau_{k}})[xz_{1}\dots z_{k}=_{0}yz_{1}\dots z_{k}]
\ee
with the type $\tau$ as above.  
Furthermore, the system $\RCAo$ includes the axiom of extensionality as follows:
\be\label{EXT}\tag{\textsf{E}}  
(\forall \varphi^{\rho\di \tau})(\forall  x^{\rho},y^{\rho}) \big[x=_{\rho} y \di \varphi(x)=_{\tau}\varphi(y)   \big].
\ee
However, as noted in \cite{brie}*{p.\ 1973}, the axiom of \emph{standard extensionality} \eqref{EXT}$^{\st}$ cannot be included in the system $\P$ (and hence $\RCAO$).  
Finally, a functional $\Xi^{2}$ is called an \emph{extensionality functional} for $\varphi^{1\di 1}$ if 
\be\label{turki}
(\forall k^{0}, f^{1}, g^{1})\big[ \overline{f}\Xi(f,g, k)=_{0}\overline{g}\Xi(f,g,k) \di \overline{\varphi(f)}k=_{0}\overline{\varphi(g)}k \big].  
\ee
In other words, $\Xi$ witnesses \eqref{EXT} for $\Phi$.  As will become clear in Section \ref{tempie}, standard extensionality is translated by our algorithm $\RS$ into the existence of an extensionality functional, and the latter amounts to merely an unbounded search.   
\end{rem}

\subsection{General template}\label{tempie}  
In this secton, we formulate a general template for obtaining explicit equivalences between arithmetical comprehension and uniform versions of principles from the RM zoo.  
This template was first formulated in \cite{samzoo} and will be applied to a number of new members of the RM zoo in Section \ref{CLASS}; it will be refined to systems of intuitionistic logic in Section \ref{culdesac}.    

\medskip

First of all, the notion of explicit implication is defined as follows.  
\bdefi[Explicit implication]\label{himplicit}
An implication $(\exists \Phi)A(\Phi)\di (\exists \Psi)B(\Psi)$ is \emph{explicit} if there is a term $t$ in the language such that additionally $(\forall \Phi)[A(\Phi)\di B(t(\Phi))]$, i.e.\ $\Psi$ can be explicitly defined in terms of $\Phi$.  
\edefi
Given that an extensionality functional as defined in Remark \ref{equ} amounts to nothing more than an unbounded search, an implication as in the previous definition will still be called `explicit' if the term $t$ additionally involves an extensionality functional $\Xi$ for $\Phi$ as defined in \eqref{turki}.  

\medskip

Secondly, we need the following functional version of arithmetical comprehension, called \emph{Feferman's non-constructive search operator} (See e.g.\ \cite{avi2}*{\S8.2}):
\be\tag{$\mu^{2}$}
(\exists^{2}\mu)(\forall f^{1})\big( (\exists n^{0})f(n)=0 \di f(\mu(f))=0 \big), 
\ee
equivalent to $(\exists^{2})$ over $\RCAo$ by \cite{kohlenbach2}*{Prop.\ 3.9}.  
We also require the following special case of the $\textsf{IST}$ axiom \emph{Transfer}.  
\be\tag{$\paai$}
(\forall^{\st}f^{1})\big( (\exists n^{0})f(n)=0 \di (\exists^{\st} m^{0})f(m)=0  \big).  
\ee
Thirdly, with these definitions in place, our template is formulated as follows.
\begin{tempo*}\rm
Let $T\equiv (\forall X^{1})(\exists Y^{1})\varphi(X,Y)$ be a  RM zoo principle and let $UT$ be $(\exists \Phi^{1\di 1})(\forall X^{1})\varphi(X,\Phi(X))$.  
To prove the \emph{explicit} implication $ UT\di (\mu^{2})$, execute the following steps:   
\begin{enumerate}
\item[(i)] Let $UT^{+}$ be $(\exists^{\st} \Phi^{1\di 1})(\forall^{\st} X^{1})\varphi(X,\Phi(X))$ where the functional $\Phi$ is additionally standard extensional.  We work in $\RCA_{0}^{\Lambda}+UT^{+}$.
\item[(ii)] Suppose the standard function $h^{1}$ is such that $(\forall^{\st}n)h(n)=0$ and $(\exists m)h(m)\ne0$, i.e.\ $h$ is a counterexample to $\paai$.  
\item[(iii)] For standard $V^{1}$, use $h$ to define standard $W^{1}\approx_{1} V$ such that $\Phi(W)\not\approx_{1}\Phi(V)$, i.e.\ $W$ is $V$ with the nonstandard elements changed sufficiently to yield a different image under $\Phi$.  \label{itemf}
\item[(iv)] The previous contradiction implies that $\RCA_{0}^{\Lambda}$ proves $UT^{+}\di \paai$.\label{forkiiii}
\item[(v)] Bring the implication from the previous step into the normal form\\ $(\forall^{\st}x)(\exists^{\st}y)\psi(x,y)$ ($\psi$ internal) using the algorithm $\mathcal{B}$ from Remark \ref{algob}.  \label{frink}
\item[(vi)] Apply the term extraction algorithm $\mathcal{A}$ using Corollary \ref{consresultcor2}.  The resulting term yields the explicit implication $UT\di (\mu^{2})$.  \label{frink2}
\end{enumerate}  
The explicit implication $(\mu^{2})\di UT$ is usually straightforward;  alternatively, establish $\paai\di UT^{+}$ in $\RCA_{0}^{\Lambda}$ and apply steps \eqref{frink} and \eqref{frink2}.  
\end{tempo*}
The algorithm $\RS$ is defined as the steps \eqref{frink} and \eqref{frink2} in the template, i.e.\ the application of the algorithms $\mathcal{B}$ and $\mathcal{A}$ to suitable implications.  

\medskip

By way of example, the following theorem was established in \cite{samzoo}*{\S3}, where $\UDNR(\Psi)$ and $\MU(\mu)$ are $\UDNR$ and $(\mu^{2})$ without the leading existential quantifier.   
\begin{thm}\label{sef}
From the proof of $\UDNR^{+}\asa \paai$ in $\RCAO $, two terms $s, u$ can be extracted such that $\RCAo$ proves:
\be\label{frood}
(\forall \mu^{2})\big[\textsf{\MU}(\mu)\di \UDNR(s(\mu)) \big] \wedge (\forall \Psi^{1\di 1})\big[ \UDNR(\Psi)\di  \MU(u(\Psi, \Phi))  \big],
\ee
where $\Phi$ is an extensionality functional for $\Psi$.
\end{thm}
From this theorem, we may conclude that $\RCAo$ proves $\UDNR\asa (\mu^{2})$, and that this equivalence is `explicit' as in Definition \ref{himplicit}.  

\medskip

Finally, the above template treats zoo-principles in a kind of `$\Pi_{2}^{1}$-normal form', for the simple reason that most zoo-principles are formulated in such a way.  
Nonetheless, it is a natural question, discussed in \cite{samzoo}*{\S6}, whether principles \emph{not} formulated in this normal form gives rise to uniform principles not equivalent to $(\mu^{2})$.    
Surprisingly, the answer to this question turned out to be negative.

\section{Classifying the RM zoo}\label{CLASS}
In this section, we apply the template from Section \ref{tempie} to a number of new theorems from the RM zoo.  
In each case, we show that the uniform version of the RM zoo principle is (explicitly) equivalent to arithmetical comprehension.  
\subsection{Universal genericity}
In this section, we study the principle $\PG$ from \cite{AMT}*{\S4} and \cite{dslice}*{Def.\ 9.44}, which is the statement that \emph{for every collection of uniformly $\Pi_{1}^{0}$ dense predicates on $2^{<\N}$, there is a sequence in $2^{\N}$ meeting all predicates}.  Like in \cite{brie}, we use the notation $\sigma^{0^{*}}\leq_{0^{*}}1$ to denote that $\sigma$ is a finite binary sequence.   
\begin{princ}[$\PG$]
Define $D_{i}(\sigma)\equiv \varphi(i, \sigma)$ with $\varphi\in \Pi_{1}^{0}$.  We have  
\[
(\forall i^{0})(\forall \tau^{0^{*}}\leq_{0^{*}}1)(\exists \sigma^{0^{*}}\succeq \tau)D_{i}(\sigma)\di (\exists G^{1}\leq_{1}1)(\forall i^{0})(\exists \sigma^{0} \prec G)D_{i}(\sigma).
\]
\end{princ}
The `fully' uniform version of $\PG$ is then defined as follows.  Note the function $g^{1}$ which realises the antecedent of $\PG$ and the function $\Phi(f,g)(2)$ which realises the numerical quantifier in the consequent of $\PG$.  
\begin{princ}[$\UPG$]
Define $D_{i}^{f}(\sigma)\equiv (\forall k^{0})f(k,i, \sigma)\ne 0$. There is $\Phi^{(1\times 1)\di (1\times 1)}$ such that for all $f^{1}, g^{1}$  
\begin{align}
(\forall i^{0})(\forall \tau^{0^{*}}\leq_{0^{*}}1)\big[ g(i,\tau)&\succeq \tau \wedge D^{f}_{i}\big(g(i,\tau)\big)\big]\label{frok}\\
&\di (\forall i^{0})\big[ \Phi(f, g)(2)(i) \prec \Phi(f,g)(1)) \wedge D^{f}_{i}\big(\Phi(f,g)(2)(i)\big)\big].\notag
\end{align}
\end{princ}
\begin{thm}\label{plugi}
In $\RCAo$, we have $\UPG\asa (\mu^{2})$, and this equivalence is explicit.  
\end{thm}
\begin{proof}
The reverse implication is immediate as $\ACA_{0}$ implies $\UPG$ and $(\mu^{2})$ easily (and explicitly) yields $\UPG$ as all relevant notions are arithmetical.  
We will now apply the template from Section \ref{tempie} to obtain the explicit implication $\UPG\di (\mu^{2})$.  

\medskip

Working in $\RCAO +\UPG^{+}$, suppose $\neg\paai$, i.e.\ there is a function $h$ such that $(\forall^{\st}n^{0})h(n)=0\wedge (\exists m^{0})h(m)\ne 0$.        
Recall from Section \ref{tempie} that $\UPG^{+}$ expresses that $\UPG$ holds, and the functional $\Phi$ in the latter is standard and standard extensional.  
Now let $f^{1},g^{1}$ be standard functions such that the antecedent of \eqref{frok} holds.  Define the standard function $g_{0}$ as follows:  
\be\label{dorga}
g_{0}(i, \tau):=
\begin{cases}
g(i,\tau*\langle k\rangle) & \tau \prec \Phi(f, g)(1)\wedge (\exists n\leq |\tau|)h(n)\ne0~\wedge  \\
~&  k\leq 1 \textup{ is the least such that } \tau*\langle k\rangle \not\prec \Phi(f, g)(1) \\
g(i, \tau ) & \text{otherwise}
\end{cases}
\ee
By the definition of $g_{0}$, we still have $(\forall i^{0})(\forall \tau^{0^{*}}\leq_{0^{*}}1)\big[ g_{0}(i,\tau)\succeq \tau \wedge D^{f}_{i}\big(g_{0}(i,\tau)\big)\big]$.     
Furthermore define the standard function $f_{0}$ as follows:
\[
f_{0}(k,i, \tau):=
\begin{cases}
f(k, i,\tau) & (\forall n\leq \max(|\tau|, i, k))(h(n)= 0) \vee \tau \not\prec \Phi(f, g)(1) \\
0 & \text{otherwise}
\end{cases}
\]
Intuitively speaking, $f_{0}$ is just $f$ with (long enough) initial segments of $\Phi(f, g)(1)$ mapping to zero.  
Nonetheless, by the definition of $f_{0}, g_{0}$, we still have 
\[
(\forall i^{0})(\forall \tau^{0}\leq_{0}1)\big[ g_{0}(i,\tau)\succeq \tau \wedge D^{f_{0}}_{i}\big(g_{0}(i,\tau)\big)\big],  
\]
as the modification to $g$ in \eqref{dorga} is such that `too long' initial segments of $\Phi(f,g)(1)$ are never output by $g_{0}$.  
Since $f\approx_{1} f_{0}$ and $g\approx_{1} g_{0}$, standard extensionality implies: 
\be\label{EX5}
\Phi(f, g)\approx_{1\times 1}\Phi(f, g_{0})\approx_{1\times 1}\Phi(f_{0}, g)\approx_{1\times 1}\Phi(f_{0}, g_{0}).
\ee
Applying $\UPG$ for $f_{0}, g_{0}$, we obtain for any $i$:
\be\label{gendgama}
\Phi(f_{0}, g_{0})(2)(i) \prec \Phi(f_{0},g_{0})(1)) \wedge D^{f_{0}}_{i}\big(\Phi(f_{0},g_{0})(2)(i)\big),
\ee
and by standard extensionality \eqref{EX5}, we have $\Phi(f_{0}, g_{0})(2)(i) =_{0}\Phi(f, g)(2)(i)$ and also $ \Phi(f, g)(1)\approx_{1} \Phi(f_{0}, g_{0})(1)$ for standard $i$.  
However, now consider the second conjunct of \eqref{gendgama}, which is $(\forall k^{0})f_{0}(k,i,\Phi(f_{0}, g_{0})(2)(i))\ne 0$.  For large enough $k$ and \emph{standard} $i$, we are in the \emph{second case} of the definition of $f_{0}$ as 
$\Phi(f_{0}, g_{0})(2)(i)\prec \Phi(f_{0}, g_{0})(1)\approx_{1} \Phi(f,g)(1)$, by standard extensionality, the first conjunct of \eqref{gendgama}, and the fact that $\Phi(f_{0}, g_{0})(2)(i)$ is standard.  
However, the second conjunct of \eqref{gendgama} contradicts the second case of $f_{0}$, and this contradiction implies $\paai$.  

\medskip

Hence, we have established $\UPG^{+}\di\paai $ inside $\RCAO$.  We now bring the former implication into normal form.  First of all, note that $\paai$ implies 
\be\label{curk}
(\forall^{\st}f^{1})(\exists^{\st}m^{0})\big[ (\exists n^{0})f(n)\ne0 \di (\exists i\leq m)f(i)\ne0  \big], 
\ee
which is a normal form, and where $C(f, m)$ is the formula in square brackets in \eqref{curk}.  Furthermore, $\UPG^{+}$ has the form
\be\label{lobia}
(\exists^{\st}\Phi)\big[(\forall^{\st}f^{1}, g^{1})A(f,g, \Phi)\wedge \Phi \textup{ is standard extensional}\big], 
\ee
where $A(f,g, \Phi)$ is exactly \eqref{frok}.  The second conjunct of \eqref{lobia} is:
\[
(\forall^{\st}f^{1}, g^{1}, u^{1}, v^{1})\big(u\approx_{1}v\wedge f\approx_{1} g\di \Phi(f,g)\approx_{1\times 1}\Phi(u,v)\big).
\]
Resolving all instances of `$\approx_{\rho}$', we obtain that for all standard $f^{1}, g^{1}, u^{1}, v^{1}$:
\[
(\forall^{\st}N^{0})(\overline{u}N=_{0}\overline{v}N\wedge \overline{f}N=_{0} \overline{g}N)\di (\forall i\leq 1)(\forall^{\st}k^{0})(\overline{\Phi(f,g)(i)}k=_{0}\overline{\Phi(u,v)(i)}k).
\]
Bringing all standard quantifiers outside, we obtain 
\be\label{finkal}
(\forall^{\st}f^{1}, g^{1}, u^{1}, v^{1}, k^{0}, i^{0}\leq 1)(\exists^{\st}N^{0})B(f,g,u,v,k,i,N, \Phi), 
\ee
where $B$ is the formula
\be\label{tokamak}
(\overline{u}N=_{0}\overline{v}N\wedge \overline{f}N=_{0} \overline{g}N)\di (\overline{\Phi(f,g)(i)}k=_{0}\overline{\Phi(u,v)(i)}k).
\ee
Combining all the previous, $\UPG^{+}\di\paai $ implies that 
\begin{align*}
\big[ (\exists^{\st}\Phi)\big[(\forall^{\st}h^{1}, g^{1})A(h,g, \Phi)\wedge &  (\forall^{\st}Z^{1})(\exists^{\st}N^{0})B(Z,N, \Phi) \big]\di (\forall^{\st}f^{1})(\exists^{\st}m^{0})C(f,m), 
\end{align*}
where $Z^{1}$ codes all the variables in the leading quantifier of \eqref{finkal}.
This yields
\begin{align*}
(\forall^{\st}\Phi, \Xi)\Big[ \big[(\forall^{\st}h^{1}, g^{1})A(h,g, \Phi)\wedge &  (\forall^{\st}Z^{1})B(Z, \Xi(Z), \Phi) \big]\di (\forall^{\st}f^{1})(\exists^{\st}m^{0})C(f,m)\Big], 
\end{align*}
and dropping some `st' and bringing all standard quantifiers to the front:
\be\label{structure}
(\forall^{\st}\Phi, \Xi, f)(\exists^{\st}m^{0})\Big[ \big[(\forall h^{1}, g^{1})A(h,g, \Phi)\wedge   (\forall Z^{1})B(Z, \Xi(Z), \Phi) \big]\di C(f,m)\Big], 
\ee
which is a normal form provable in $\RCAO$.  Applying Corollary \ref{consresultcor2}, there is a term $t$ such that $\RCAo$ proves 
\begin{align*}
(\forall \Phi, \Xi, f)(\exists m^{0}\in t(\Phi, \Xi, f))\Big[ \big[(\forall h^{1}, g^{1})A(h,g, \Phi)\wedge &  (\forall Z^{1})B(Z, \Xi(Z), \Phi) \big]\di C(f,m)\Big], 
\end{align*}
where $\Phi$ is as in $\UPG$ by $(\forall h, g)A(h,g, \Phi)$ and $\Xi$ is the associated extensionality functional by $(\forall Z^{1})B(Z, \Xi(Z), \Phi)$.  
Now define $s(\Phi, \Xi, f)$ as the maximum of all $t(\Phi, \Xi, f)$ for $i<|t(\Phi, \Xi, f)|$ and note that $(\forall f^{1})C(f, s(\Phi,\Xi, f))$ expresses that $s(\Phi,\Xi, f)$ is Feferman's non-constructive search operator.  In other words, we have obtained the explicit implication $\UPG\di (\mu^{2})$, and we are done.  
\end{proof}
We immediately obtain the following `more explicit' corollary, where $\UPG(\Phi)$ is just $\UPG$ with the leading existential quantifier omitted.  
\begin{cor}\label{sef2}
From the proof of $\UPG^{+}\asa \paai$ in $\RCAO $, two terms $s, u$ can be extracted such that $\RCAo$ proves:
\be\label{frood2}
(\forall \mu^{2})\big[\textsf{\MU}(\mu)\di \UPG(s(\mu)) \big] \wedge (\forall \Phi)\big[ \UPG(\Phi)\di  \MU(u(\Phi, \Xi))  \big],
\ee
where $\Xi$ is an extensionality functional for $\Phi$.
\end{cor}
\begin{proof}
The second conjunct is immediate from the theorem.  The first conjunct can be obtained by establishing $\paai\di \UPG^{+}$ (which is almost trivial) in $\RCAO$ and applying Corollary \ref{consresultcor2} to this implication in normal form.  
\end{proof}
The proof of the theorem also provides a template as follows.  
\begin{rem}[Algorithm $\mathcal{B}$]\label{algob}\rm
Let $T\equiv (\forall X^{1})(\exists Y^{1})\varphi(X, Y)$ be an internal formula and define the `strong' uniform version $UT^{+}$ as 
\[
(\exists^{\st} \Phi^{1\di 1})\big[(\forall^{\st} X^{1})\varphi(X,\Phi(X)) \wedge \textup{$\Phi$ is standard extensional}\big].
\]
The proof of Theorem \ref{plugi} provides a \emph{normal form algorithm} $\mathcal{B}$ to convert the implication $UT^{+}\di \paai$ into a normal form $(\forall^{\st}x)(\exists^{\st}y)\varphi(x, y)$ as in \eqref{structure}. 
\end{rem}
The previous theorem implies that we may extract an explicit equivalence from a nonstandard one.  It is then a natural question (especially in the light of Reverse Mathematics) if we can 
also re-obtain the (proof of the) nonstandard equivalence from the (proof of the) explicit equivalence.  This question will be answered in the positive in Section \ref{herbie}.

\subsection{The finite intersection principle and 1-genericity}
In this section, we study uniform versions of the finite intersection principle \textsf{FIP} from \cite{damu} and the principle 1-\textsf{GEN} related to Cohen forcing from \cite{igusa}.  
By \cite{igusa}*{Theorem 5.8}, the aforementioned principles are equivalent over $\RCA_{0}$.  

\medskip

First of all, to study \textsf{1-GEN} in the higher-order framework, we define $\sigma^{0}\in S_{f}^{X}$ as $ (\exists \tau^{0})f(\sigma, \tau, \overline{X}|\tau|)=0$ 
and let 1-$\textsf{GEN}$ and its uniform version be as follows.
\begin{princ}[1-\textsf{GEN}]
\[
(\forall X^{1})(\exists Y^{1})(\forall f^{1})\big[(\exists n^{0})(\overline{Y}n\in S_{f}^{X})\vee (\exists m^{0})(\forall \sigma\succeq \overline{Y}m)(\sigma \not\in S_{f}^{X})  \big].
\]  
\end{princ}
\begin{princ}[\textsf{U1G}]
There is $ \Phi^{1\di (1\times 2\times 2)}$ such that for all $X^{1}, f^{1}$, we have
\be\label{fras}
\big(\overline{\Phi(X)(1)}\Phi(X)(2)(f)\in S_{f}^{X}\big)\vee (\forall \sigma\succeq \overline{\Phi(X)(1)}\Phi(X)(3)(f))(\sigma \not\in S_{f}^{X}).
\ee
\end{princ}
Note that the witnessing functional in the first disjunct is actually superfluous, as the base theory includes $\QFAC^{1,0}$.  
We have the following theorem.  
\begin{thm}
In $\RCAo$, we have $\UOG\asa (\mu^{2})$, and this equivalence is explicit.  
\end{thm}
\begin{proof}
The reverse implication is immediate as $\ACA_{0}$ implies \textsf{1-GEN} and $(\mu^{2})$ easily (and explicitly) yields $\UOG$ in light of e.g.\ \cite{dohi}*{2.24.2}.  
We now prove the remaining explicit implication using the template from Section \ref{tempie}.  
Thus, working in $\RCAO+\UOG^{+}$, suppose $\neg\paai$, i.e.\ there is a function $h$ such that $(\forall^{\st}n^{0})h(n)=0\wedge (\exists m^{0})h(m)\ne 0$.  
Let $f_{0}$ and $X_{0}$ be standard sequences such that the first conjunct of \eqref{fras} is false, and define the standard function $f_{1}$ as:
\[
f_{1}(\sigma,\tau, \rho):=
\begin{cases}
f_{0}(\sigma,\tau,\rho) &\textup{otherwise} \\
0 & (\exists n \leq |\sigma|)(h(n)\ne 0)
\end{cases}.
\]
With this definition, we observe that
\[
\sigma_{1}:=\overline{\Phi(X_{0})(1)}\Phi(X_{0})(2)(f_{1})=_{0}\overline{\Phi(X_{0})(1)}\Phi(X_{0})(2)(f_{0})=: \sigma_{0},
\]
by standard extensionality, implying the following sequence of equivalences:
\begin{align*}
[\sigma_{1}\in S_{f_{1}}^{X_{0}}]&\equiv[ (\exists \tau^{0})f_{1}(\sigma_{1}, \tau, \overline{X}|\tau|)=0]\\
&\equiv[ (\exists \tau^{0})f_{0}(\sigma_{1}, \tau, \overline{X}|\tau|)=0]\equiv [ (\exists \tau^{0})f_{0}(\sigma_{0}, \tau, \overline{X}|\tau|)=0]\equiv [\sigma_{0}\in S_{f_{0}}^{X_{0}}],
\end{align*}
where  the second step holds by the definition of $f_{1}$ and the fact that $\sigma_{1}$ is standard.    
In other words, the first conjunct of \eqref{fras} is false for $X_{0}, f_{1}$.  Hence, the second conjunct of \eqref{fras} must hold for $X_{0}$ and for $ f_{1}$, i.e.\ we have 
\[
(\forall \sigma\succeq \overline{\Phi(X_{0})(1)}\Phi(X_{0})(3)(f_{1}))(\forall \tau^{0})f_{1}(\sigma, \tau, \overline{X_{0}}|\tau|)\ne 0.    
\]
Since $\overline{\Phi(X_{0})(1)}\Phi(X_{0})(3)(f_{1})$ is standard, we can apply the previous for $\sigma=\overline{\Phi(X_{0})(1)}M$ for any nonstandard $M$.  However, this yields a contradiction as $f_{1}$ is zero for long enough $\sigma$.  From this contradiction, we conclude that $\RCAO$ proves $ \UOG^{+}\di \paai$.  
Analogous to the proof of Theorem \ref{plugi}, $\UOG^{+}\di \paai$ may be brought into a normal form of the form \eqref{structure}, and applying Corollary \ref{consresultcor2} now finishes the proof.  
\end{proof}
Secondly, we briefly study the principle \textsf{FIP} in the following remark. 
\begin{rem}\rm
By \cite{damu}*{Prop.\ 2.3}, $\ACA_{0}$ is equivalent to a \emph{strengthened} version of \textsf{FIP} 
where a set $I$ is given such that $i\in I\asa A_{i}\in \mathfrak{B}$, where the latter is the maximal subfamily with the finite intersection property.  
It is straightforward to prove a uniform version (involving $(\mu^{2})$) of this equivalence.

\medskip

However, the uniform version of \textsf{FIP} will provide such a set $I$ as in the strengthened version of \textsf{FIP}.  In other words, the aforementioned results 
immediately imply that the uniform version of \textsf{FIP} is equivalent to $(\mu^{2})$.  Similarly, \cite{damu}*{Prop.~2.3} implies that the uniform versions of $n\textsf{IP}$ ($n\geq 2$) are all equivalent to $(\mu^{2})$.      
\end{rem}

\subsection{The omitting partial types principle}
In this section, we study uniform versions of the \emph{Omitting Partial Types} principle \textsf{OPT} which may be found in \cite{AMT}*{\S5} and \cite{dslice}*{Def.\ 9.64}.   

\medskip

In light of \cite{AMT}*{Theorems 5.6-5.7} and particularly \cite{dslice}*{9.66-9.67}, the uniform versions of \textsf{OPT} and \textsf{HYP} are (explicitly) equivalent.  
Hence, we study the latter, which is essentially the 
statement that \emph{for every set $X^{1}$, there is a function which is not dominated by any $X$-computable function}.  In symbols, we have 
\be\tag{\textup{\textsf{HYP}}}
(\forall f^{1})(\exists g^{1})(\forall e^{0}, k^{0})(\exists n^{0}\geq k)(\forall m^{0}, s^{0})\big[\varphi_{e,s}^{f}(n)=m \di m < g(n) \big],     
\ee
following the definition in \cite{zweer}*{p.\ 189, 3.7}.  Hence, the uniform version is
\begin{align}\tag{\textup{\textsf{UHYP}}}
(\exists \Phi^{1\di (1\times 1)})(\forall f^{1})&(\forall e^{0}, k^{0}, m^{0}, s^{0})\big[\Phi(f)(2)(e,k)\geq k~\wedge\\
& \varphi_{e,s}^{f}(\Phi(f)(2)(e,k))=m \di m < \Phi(f)(1)\big(\Phi(f)(2)(e,k)\big) \big].     \notag
\end{align}
\begin{thm}
In $\RCAo$, we have $\UHYP\asa (\mu^{2})$, and the equivalence is explicit.  
\end{thm}
\begin{proof}
The reverse implication is trivial as $(\mu^{2})$ can check if a given Turing machine halts, and avoid the output if necessary.  
Working in $\RCAO+\UHYP^{+}$, suppose $\neg\paai$, i.e.\ there is a function $h$ such that $(\forall^{\st}n^{0})h(n)=0\wedge (\exists m^{0})h(m)\ne 0$.    

\medskip

First of all, let the \emph{standard} number $e_{0}$ be the code of the following program for $\varphi_{e_{0}}^{f}$:  
On input $n$, set $k=n$ and check if $f(k)>0$;  If so, return this number;  If $f(k)=0$, repeat for $k+1$.  
Intuitively speaking, $e_{0}$ is such that $\varphi_{e_{0}}^{f}(n)$ outputs $m>0$ if starting at $k=n$, we eventually find $m=f(k)>0$, and undefined otherwise. 
Furthermore, let $f_{0}$ be the sequence $00\dots$ and define
\[
f(e):=
\begin{cases}
 \Phi(f_{0})(1)\big(\Phi(f_{0})(2)(e_{0}, e_{0})\big) &  (\exists s\leq e)h(s)\ne 0 \\
 0 & \textup{otherwise}
 \end{cases}, 
\]   
where $h$ is the exception to $\paai$ from the first paragraph of this proof.  
Note that $f\approx_{1}f_{0}$ by definition, implying that $\Phi$ satisfies $\Phi(f)\approx_{1\times 1} \Phi(f_{0})$ due to standard extensionality.  
However, the latter combined with $\UHYP$ gives us:
\begin{align}
\Phi(f_{0})(1)\big(\Phi(f_{0})(2)(e_{0}, e_{0})\big)
&=_{0}\Phi(f)(1)\big(\Phi(f)(2)(e_{0}, e_{0})\big)\label{ploef2}\\
&>_{0} \varphi_{e_{0}, s_{0}}^{f}\big(\Phi(f)(2)(e_{0}, e_{0})\big)\notag\\
&=_{0} \Phi(f_{0})(1)\big(\Phi(f_{0})(2)(e_{0}, e_{0})\big),\notag
\end{align}
for large enough $s_{0}$ such that $(\exists i\leq s_{0})h(i)\ne 0$.  Note that it is essential for the first step in \eqref{ploef2} that $\Phi(f)(2)(e_{0},e_{0})$ and $\Phi(f)(1)(\cdot)$ are \emph{standard}.  
The contradiction in \eqref{ploef2} implies that $\UHYP^{+}\di \paai$ in $\RCAO$.  
Now bring this implication in normal form and apply Corollary \ref{consresultcor2} to obtain the explicit implication.  
\end{proof}
In light of the proof of \cite{dslice}*{9.66}, the uniform version of the \emph{atomic model theorem} \textsf{AMT} is also (explicitly) equivalent to $(\mu^{2})$ by the previous theorem.  
Similarly, the proof of $\SADS\di \AMT$ in \cite{AMT}*{Theorem 4.1} is sufficiently uniform to (explicitly) yield $\USADS\di \textup{\textsf{UAMT}}$ in $\RCAo$.

\subsection{Non-computable sets}
In this section, we study the uniform version of a principle `very close to $\RCA_{0}$' in the RM-zoo.  In particular, Hirschfeldt states in \cite{dslice}*{p.~174} that the principle $\AST$ (See \cite{dslice}*{Def.\ 9.71}) is essentially the weakest principle in the zoo, in light of its equivalence to $\textsf{NCS}\equiv(\forall X^{1})(\exists Y^{1})(Y\not\leq_{T}X)$ by \cite{AMT}*{Theorem 6.3}.  
The proof of the latter is sufficiently uniform to yield the equivalence between the uniform versions of $\AST$ and \textsf{NCS}.     
Thus, we study the existence of non-computable sets as follows:  
\be\tag{$\NCS$}
 (\forall f^{1})(\exists g^{1})(\forall e^{0})(\exists n^{0})(\forall s^{0})[g(n)\ne_{0} \varphi_{e,s}^{f}(n)], 
\ee 
which has the following uniform version:
\be\tag{$\UNCS$}
(\exists \Phi^{1\di (1\times 1)})(\forall f^{1})(\forall e^{0}, s^{0})\big[ \Phi(f)(1)\big(\Phi(f)(2)(e)\big)\ne_{0} \varphi_{e, s}^{f}(\Phi(f)(2)(e))   \big].
\ee
\begin{thm}
In $\RCAo$, we have $\UNCS\asa (\mu^{2})$ and this equivalence is explicit.
\end{thm}
\begin{proof}
The explicit implication $(\mu^{2})\di \UNCS$  is trivial as $(\mu^{2})$ supplies the Turing jump of any set.  
Working in $\RCAO+\UNCS^{+}$, suppose $\neg\paai$, i.e.\ there is a function $h$ such that $(\forall^{\st}n^{0})h(n)=0\wedge (\exists m^{0})h(m)\ne 0$.    

\medskip

First of all, fix a standard pairing function $\pi^{1}$ and its inverse $\xi^{1}$.  Now let the \emph{standard} number $e_{1}$ be the code of the following program:  
On input $n$, set $k=n$ and check if $k\in A$ and if $\xi(k)(2)>0$;  If so, return this non-zero component;  If $k\not\in A$ or $\xi(k)(2)=0$, repeat for $k+1$.  Intuitively speaking, $e_{1}$ is such that $\varphi_{e_{1}}^{A}(n)$ outputs $m>0$ if starting at $k=n$, we eventually find $\pi((l,m))\in A$, and undefined otherwise. 
Furthermore, let $f_{0}$ be the sequence $00\dots$ and define
\[
f(e):=
\begin{cases}
 \Phi(f_{0})(1)\big(\Phi(f_{0})(2)(e_{1})\big) &  (\exists i\leq e)h(i)\ne 0 \\
 0 & \textup{otherwise}
 \end{cases}, 
\]   
where $h$ is the exception to $\paai$ from the first paragraph of this proof.  
Note that $f\approx_{1}f_{0}$ by definition, implying that $\Phi$ satisfies $\Phi(f)\approx_{1\times 1} \Phi(f_{0})$ due to standard extensionality.  
However, the latter combined with $\UNCS$ gives us:
\begin{align}
\Phi(f_{0})(1)\big(\Phi(f_{0})(2)(e_{1})\big)=\Phi(f)(1)\big(\Phi(f)(2)(e_{1})\big)
&\ne \varphi_{e_{1}, s_{0}}^{f}(\Phi(f)(2)(e_{1}))\label{ploef}\\
&= \Phi(f_{0})(1)\big(\Phi(f_{0})(2)(e_{1})\big),\notag
\end{align}
for large enough (infinite) $s_{0}$.  Note that it is essential for the first step in \eqref{ploef} that $\Phi(f)(2)(e_{1})$ and $\Phi(f)(1)(\cdot)$ are \emph{standard}.  
The contradiction in \eqref{ploef} implies that $\RCAO$ proves $\UNCS^{+}\di \paai$.  Now bring the latter in normal form and apply Corollary \ref{consresultcor2} to obtain the explicit implication.   
\end{proof}
Related to the above is the Kleene-Post theorem (See \cite{klepol} and \cite{zweer}*{Chapter VI}) stating the existence of (Turing) incomparable sets.  The related principle is:
\be\tag{$\KPT$}
 (\forall f^{1})(\exists g^{1}, h^{1})\big[ f\leq_{T} \langle g, h\rangle \wedge g~|_{T}~h   \big],
\ee   
We denote by $\UKPT$ the fully uniform, i.e.\ with all existential quantifiers removed, version of $\KPT$.  Clearly, $\UKPT$ implies $\UNCS$ and the equivalence $\UKPT\asa (\mu^{2})$ is now straightforward in light of \cite{zweer}*{VI.1, p.\ 93}.  


\section{Refining our results: meta-reversal and intuitionistic logic}\label{refine}
In this section, we refine some of the results from \cite{samzoo} and the previous sections.  
First of all, we derive Theorem \ref{plugi} using only systems based on intuitionistic logic in Section \ref{culdesac}.  The associated proof gives rise to a refinement of our template from Section \ref{tempie}.  Secondly, we provide a `meta-reversal' for Corollary \ref{sef2} in Section~\ref{herbie} as follows:  We show that a version of \eqref{frood2}, called the \emph{Herbrandisation} of $\UPG^{+}\di \paai$, implies the nonstandard implication $\UPG^{+}\di \paai$ from which it was obtained.  As we will see, these results generalise to all explicit equivalences proved above and in \cite{samzoo}.    
\subsection{Non-classical equivalences}\label{culdesac}
The explicit equivalences from the previous sections and \cite{samzoo} were established in $\RCAO$ and $\RCAo$, i.e.\ systems based on classical logic.  
We show in this section that Corollary~\ref{sef2} essentially goes through for systems based on intuitionistic logic.  
As will become clear, the same technique applies to all theorems in this paper and \cite{samzoo}.  

\medskip

This `constructive result' is somewhat surprising, as our hitherto obtained results seem to make essential use of non-constructive principles:  
For instance, $\UPG^{+}\di \paai$ was proved via a \emph{proof-by-contradiction}, while obtaining the normal form of this implication involves the \emph{independence of premises} principle to bring the standard quantifiers up front.  Furthermore, basic results from computability theory, like \emph{Post's complementation theorem} (\cite{zweer}*{Theorem 1.12}), 
already require non-constructive principles (\cite{troeleke1}*{\S4.5.3}), while our \emph{nonstandard} technique will turn out to have a constructive counterpart.   

\medskip

The previous observations notwithstanding, let $\H$ be the conservative extension of Heyting arithmetic introduced in \cite{brie}*{\S5.2}.  
Note that $\P$ from Section \ref{popo} is just $\H$ with classical instead of intuitionistic logic.  
We consider two axioms of $\H$, essential for the proof of Theorem \ref{sefke} below.  
\bdefi[Two axioms of $\H$]\label{flah}~
\begin{enumerate}\rm
\item $\textsf{HIP}_{\forall^{\st}}$
\[
[(\forall^{\st}x)\phi(x)\di (\exists^{\st}y)\Psi(y)]\di (\exists^{\st}y')[(\forall^{\st}x)\phi(x)\di (\exists y\in y')\Psi(y)],
\]
where $\Psi(y)$ is any formula and $\phi(x)$ is an internal formula of \textsf{E-HA}$^{\omega*}$. 
\item $\textsf{HGMP}^{\st}$
\[
[(\forall^{\st}x)\phi(x)\di \psi] \di (\exists^{\st}x')[(\forall x\in x')\phi(x)\di \psi] 
\]
where $\phi(x)$ and $\psi$ are internal formulas in the language of \textsf{E-HA}$^{\omega*}$.
\end{enumerate}
\edefi
Intuitively speaking, the two axioms of Definition \ref{flah} allow us to perform a number of \emph{non-constructive operations} (namely \emph{Markov's principle} and \emph{independence of premises}) 
on standard objects.  
In other words, the standard world of $\H$ is `a little non-constructive', but this does not affect the conservation result over Heyting arithmetic: $\H$ and $\textsf{E-HA}^{\omega}$ prove the same \emph{internal} formulas by \cite{brie}*{Cor.\ 5.6}.  

\medskip

Surprisingly, we will observe that the axioms from Definition \ref{flah} are \emph{exactly} what is needed for the proof of Corollary \ref{sef2} to go through constructively.  
As in the proof of Theorem \ref{plugi}, we shall focus on the implication $\UPG^{+}\di \paai$, while the other implication is treated analogously.  
Note that $\UPG(\Phi)$ is $\UPG$ with the leading quantifier omitted and $\MU(\mu)$ is $(\forall f^{1})\textsf{\textup{MUP}}(f, \mu)$.  
\begin{thm}\label{sefke}
From the proof of $\UPG^{+}\di \paai$ in $\H $, a term $t$ can be extracted such that $\textsf{\textup{E-HA}}^{\omega*}$ proves:
\be\label{froodke}
(\forall \Phi^{1\di 1}, f^{1})\big[ \UPG(\Phi)\di  \textsf{\textup{MUP}}(f, t(\Psi, \Phi, f))  \big],
\ee
where $\Psi$ is an extensionality functional for $\Phi$.
\end{thm}
\begin{proof}
To show that $\H$ proves $\UPG^{+}\di \paai$, it is straightforward to verify that the second part of the proof of Theorem \ref{plugi} yields that 
\be\label{conki}
\UPG^{+}\di (\forall^{\st}f^{1})\big[  (\exists n)f(n)=0 \di  \neg[(\forall^{\st}n)f(n)\ne0]\big], 
\ee
since $\H$ is based on intuitionistic logic.  However, by Definition \ref{flah}, the system $\H$ proves\footnote{Take $\psi\equiv [0=1]$ and $\phi$ a decidable formula in \textsf{HGMP}$^{\st}$ in Definition \ref{flah}.} Markov's principle relative to `st', and hence:
\be\label{boner}
\neg[(\forall^{\st}n)f(n)\ne0]\di (\exists^{\st}n)f(n)=0.
\ee
Combining \eqref{conki} and \eqref{boner}, we obtain $\UPG^{+}\di \paai$ inside $\H$.  Now, the latter sysem also has a term extraction result, namely \cite{brie}*{Theorem 5.9}, 
which is identical to Corollary~\ref{consresultcor}.  Hence, we only need to bring $\UPG^{+}\di \paai$ into a normal form like \eqref{structure} \emph{inside} $\H$, and \eqref{froodke} follows.  
We now bring $\UPG^{+}\di \paai$ into a slight variation of the normal form \eqref{structure} inside $\H$.  

\medskip

First of all,  applying the principle $\textsf{HIP}_{\forall^{\st}}$ from Definition \ref{flah} to $\paai$, the latter implies \eqref{curk}, i.e.\ the former has a normal form, 
say $(\forall^{\st}f^{1})(\exists^{\st}n^{0})C(f,n)$.  
Secondly, the second conjunct of $\UPG^{+}$ immediately implies (in $\H$) that for all standard $f^{1}, g^{1}, u^{1}, v^{1}, k^{0}, i^{0}\leq 1$, we have
\be\label{dalk}
\big((\forall^{\st}N^{0})(\overline{u}N=_{0}\overline{v}N\wedge \overline{f}N=_{0} \overline{g}N)\big)\di \overline{\Phi(f,g)(i)}k=_{0}\overline{\Phi(u,v)(i)}k, 
\ee
and applying $\textsf{HGMP}^{\st}$ to \eqref{dalk}, we obtain 
\[
(\exists^{\st}N')\big[(\forall N^{0}\leq N')(\overline{u}N=_{0}\overline{v}N\wedge \overline{f}N=_{0} \overline{g}N)\di \overline{\Phi(f,g)(i)}k=_{0}\overline{\Phi(u,v)(i)}k\big].  
\]
Define $Z^{1}$ as a code for the tuple of variables $f^{1}, g^{1}, u^{1}, v^{1}, k^{0}, i^{0}\leq 1$ and let $B(Z,N', \Phi)$ be the formula in square brackets in the previous centred formula.
Thus, the second conjunct of $\UPG^{+}$ has the normal form $(\forall^{\st}Z^{1})(\exists^{\st}M^{0})B(Z, M, \Phi)$ and $\UPG^{+}\di \paai$ implies:
\[
\big[(\exists^{\st}\Phi)(\forall^{\st}h^{1},g^{1})A(h, g, \Phi)\wedge  (\exists^{\st}\Xi^{2})(\forall^{\st}Z^{1})B(Z, \Xi(Z), \Phi)\big] \di (\forall^{\st}f^{1})(\exists^{\st}y^{0})C(f, y),
\]
where $A(\cdot)$ is \eqref{frok} and the antecedent is strengthened by introducing $\Xi$.  Inside $\H$, we can bring outside the quantifiers involving the variables $\Psi$, $\Xi$, and $f$, yielding:  
\[
(\forall^{\st}\Phi, \Xi, f)\Big(\big[(\forall^{\st}h^{1},g^{1})A(h, g, \Phi)\wedge  (\forall^{\st}Z^{1})B(Z, \Xi(Z), \Phi)\big] \di (\exists^{\st}y^{0})C(f, y)\Big),
\]
which has exactly the right syntactic structure to apply $\HIP_{\forall^{\st}}$, and we obtain:
\[
(\forall^{\st}\Phi, \Xi, f)(\exists^{\st}\sigma^{0^{*}})\Big(\big[(\forall^{\st}h^{1},g^{1})A(h, g, \Phi)\wedge  (\forall^{\st}Z^{1})B(Z, \Xi(Z), \Phi)\big] \di (\exists y^{0}\in \sigma)C(f, y)\Big),
\]
and the latter now has exactly the right structure to apply $\HGMP^{\st}$, and we obtain:
\begin{align}
&(\forall^{\st}\Phi, \Xi, f)(\exists^{\st}\sigma^{0^{*}}, W^{1^{*}}, V^{1^{*}})\label{bling}\\
&\Big(\big[(\forall h^{1},g^{1}\in W)A(h, g, \Phi)\wedge  (\forall Z^{1}\in V)B(Z, \Xi(Z),\Phi)\big] \di (\exists y^{0}\in \sigma)C(f, y)\Big),\notag
\end{align}
which is a slight variation of the normal form \eqref{structure}, and the theorem follows by applying the term extraction result from \cite{brie}*{Cor.\ 5.9}.  
\end{proof}
\begin{cor}
In $\textup{\textsf{E-HA}}^{\omega*}$, $\UPG\asa (\mu^{2})$.  
\end{cor}
Note that we could have worked in a fragment of $\H$ similar to $\RCAO$.  
We finish this section with a remark stipulating the refinement using intuitionistic logic of the template in Section \ref{tempie}.  
\begin{rem}\rm
Based on the proof of Theorem \ref{sefke}, the template from Section \ref{tempie} can be refined as follows to work for intuitionistic instead of classical logic.  
\begin{enumerate}
\item Replace $\RCAO$ and $\RCAo$ by $\H$ and $\textsf{E-HA}^{\omega*}$.  
\item In step \eqref{forkiiii} of the template, we obtain that $\H\vdash UT^{+}\di \paai$ from
\[
UT^{+}\di (\forall^{\st}f^{1})\big[  (\exists n)f(n)=0 \di  \neg[(\forall^{\st}n)f(n)\ne0]\big], 
\]
and \textsf{HGMP}$^{\st}$ as in \eqref{boner} from the proof of Theorem \ref{sefke}.  
\item In step \eqref{frink} of the template, use \textsf{HGMP}$^{\st}$ and $\textsf{HIP}_{\forall^{\st}}$ as in the proof of Theorem \ref{sefke} to obtain a normal form of $UT^{+}\di \paai$.  
\item In step \eqref{frink2} of the template, apply the term extraction result formulated in \cite{brie}*{Theorem~5.9} to the normal form of $UT^{+}\di \paai$.  
\end{enumerate}
Finally, it is surprising -in our opinion- that $\H$ includes exactly the right `non-constructive' axioms -\emph{relative to} `\st'- as in Definition \ref{flah} to make the proof of Theorem \ref{sef2} go through in a constructive setting.  
\end{rem}

\subsection{Hebrandisations}\label{herbie}
In this section, we provide a positive answer to the following natural RM-style question:  
\begin{quote}
The algorithm $\RS$ takes as input implications in Nonstandard Analysis and produces explicit implications related to the RM zoo.  
Is it possible to re-obtain these nonstandard `pre-cursor' implications from their `post-cursor' explicit implications? 
\end{quote}
To answer this question, we shall study the explicit implication $\UPG\di (\mu^{2})$ from Theorem \ref{plugi}, in particular a variation of the second conjunct of \eqref{frood2}, defined as:   
\begin{align}
(\forall \Phi, \Xi, f^{1}) \Big[\big[ & (\forall Z^{1}\in i(\Psi, \Xi, f)(1))B(Z, \Xi(Z), \Phi) \wedge  (\forall f,g^{1}\in i(\Phi, \Xi, f)(2))A(f,g,\Phi) \big]\notag\\
&\di \big( (\exists n)f(n)=0\di (\exists j\leq o(\Psi, \Xi,f ))f(j)=0\big)\Big]\tag{$\HIO(i,o)$}
\end{align}
where $A(\cdot)$ is \eqref{frok} from $\UPG$ and $B(\cdot)$ is \eqref{tokamak} and expresses that $\Xi$ is an extensionality functional for $\Phi$.  
We refer to $\HIO(i,o)$ as the \emph{Herbrandisation} of $\UPG^{+}\di \paai$.  
Intuitively speaking, the functional $i$ in the Herbrandisation tells us `how much' $\Phi$ has to satisfy $\UPG$ \emph{for a particular} $f^{1}$ in order to obtain 
the value of the mu-operator at $f$ via $o$ (and the same for $\Xi$).  In other words, the Herbrandisation is a `pointwise' version of the second conjunct of \eqref{frood2}.  

\medskip
  
We have the following theorem establishing a `meta-reversal' between the implication $\UPG^{+}\di \paai$ and its Herbrandisation $\HIO(i,o)$.    
\begin{thm}[Meta-reversal]\label{self}
From the proof of $\UPG^{+}\di \paai$ in $\RCAO $, two terms $i, o$ can be extracted such that $\RCAo$ proves $\HIO(i,o)$.  If there are terms $i, o$ such that $\RCAo$ proves $ \HIO(i,o)$, then $\RCAO$ proves $\UPG^{+}\di \paai$.    
\end{thm}
\begin{proof}
The first part of the theorem easily follows from the proof of Theorem \ref{plugi}.  Indeed, consider \eqref{structure}, but without the `st' in the antecedent dropped, as follows:
\[
(\forall^{\st}\Phi, \Xi, f)(\exists^{\st}m^{0})\Big[ \big[(\forall^{\st} h^{1}, g^{1})A(f,g, \Phi)\wedge   (\forall^{\st} Z^{1})B(Z, \Xi(Z),\Phi) \big]\di C(f,m)\Big], 
\]
which yields the following by bringing out the standard quantifiers:
\be\label{structure3}
(\forall^{\st}\Phi, \Xi, f)(\exists^{\st}m^{0}, h^{1}, g^{1}, Z^{1})\Big[ \big[A(h,g, \Phi)\wedge   B(Z, \Xi(Z),\Phi) \big]\di C(f,m)\Big], 
\ee 
Apply Corollary \ref{consresultcor2} to `$\RCAO\vdash \eqref{structure3}$' to obtain a term $t$ such that $\RCAo$ proves
\[
(\forall \Phi, \Xi, f)(\exists m^{0}, h^{1}, g^{1}, Z^{1}\in t(\Phi, \Xi, f))\Big[ \big[A(h,g, \Phi)\wedge   B(Z, \Xi(Z),\Phi) \big]\di C(f,m)\Big], 
\] 
Define the term $o$ as the maximum of all entries of $t$ pertaining to $m$;  define $i(\Psi, \Xi, f)(i)$ for $i=1$ (resp.\ $i=2$) as all entries of $t$ pertaining to $h,g$ (resp.\ pertaining to $Z$).   
Then $\HIO(i, o)$ follows and this part is done.  

\medskip

For the second part of the theorem, suppose $i, o$ are terms such that $\RCAo$ proves $\HIO(i,o)$.  By the second standardness axiom (See Definition \ref{debs}), these terms are standard in $\RCAO$, i.e.\ the latter proves 
$\HIO(i,o)\wedge\st(i)\wedge\st(o)$.  Hence, for standard $\Phi, \Xi, f$, the terms $o(\Phi, \Xi, f)$ and $i(\Phi, \Xi, f)$ are standard (by the third standardness axiom in Definition \ref{debs}), and the consequent of $\HIO(i, o)$ clearly yields $\paai$, while the antecedent of the $\HIO(i, o)$ holds if $\UPG^{+}$ does.      
Thus, we obtain that $\RCAO$ proves $\UPG^{+}\di \paai$, and we are done.  
\end{proof}
Similar results hold for the first disjunct in \eqref{frood2}.  In general, one can obtain the Herbrandisation for any nonstandard equivalence from this paper and \cite{samzoo}, and prove a result 
similar to the previous theorem.  Intuitively speaking, the nonstandard implication $\UPG^{+}\di \paai$ and its Herbrandisation $\HIO(i,o)$ can be said to be `meta-equivalent' or `share the same computational content' in the sense of the theorem, namely that one can be obtained from the latter via an algorithmic manipulation.

\begin{ack}\rm
This research was supported by the following funding bodies: FWO Flanders, the John Templeton Foundation, 
the Alexander von Humboldt Foundation, LMU Munich, and the Japan Society for the Promotion of Science.  
The author expresses his gratitude towards these institutions.  The author would also like to thank the anonymous referee who has offered a number of suggestions which greatly improved this paper.     
\end{ack}

\bye